\documentclass[12pt]{article}

\usepackage{amsmath}
\usepackage{amsthm}
\usepackage{amssymb}
\usepackage{float}    
\usepackage{verbatim} 
\usepackage{mathtools} 
\usepackage{times}
\usepackage{enumerate}
\usepackage{paralist}
\usepackage{xspace}
\usepackage{bbm}
\usepackage{mathrsfs}
\usepackage{thmtools}
\usepackage{thm-restate}
\usepackage{dsfont}
\usepackage[square, numbers]{natbib}
\usepackage{fullpage}

\newtheorem{prop}{Proposition}

\newtheorem{theorem}{Theorem}
\newtheorem{corollary}{Corollary}
\newtheorem{lemma}{Lemma}
\newtheorem*{thm*}{Main result (informal)}

\DeclareMathOperator*{\argmin}{arg\!\min}

\DeclareMathOperator*{\tr}{\mathbf{tr}}

\newcommand{\Ncal}{\ensuremath{\mathcal{N}}}
\newcommand{\Ecal}{\ensuremath{\mathcal{E}}}

\newcommand{\Ocal}{\ensuremath{\mathcal{O}}}

\newcommand{\Sbb}{\ensuremath{\mathbb{S}}}

\newcommand{\EE}{\mathbb{E}}

\newcommand{\Sym}{\mathbf{S}}

\newcommand{\PP}{\mathbb{P}}
\newcommand{\T}{\mathsf{T}}

\newcommand{\RR}{\mathbb{R}}

\newcommand{\vecx}{\mathbf{x}}
\newcommand{\vecb}{\boldsymbol{\beta}}

\newcommand{\vecy}{\mathbf{y}}

\newcommand{\statedim}{d}
\newcommand{\outdim}{d}

\newcommand{\Ahat}{\widehat{A}}

\newcommand{\Mhat}{\widehat{M}}

\newcommand{\Astar}{A_\star}
\newcommand{\Bstar}{B_\star}
\newcommand{\Cstar}{C_\star}

\newcommand{\iid}{\stackrel{\mathclap{\text{\scriptsize{ \tiny i.i.d.}}}}{\sim}}

\newcommand{\that}{\widehat{\theta}}

\newcommand{\sshat}{\widehat{\Sigma}_\infty}

\newcommand{\sscov}{{\Sigma}_\infty}

\renewcommand{\vec}{\operatorname{vec}}
\newcommand{\svec}{\operatorname{svec}}

\newcommand{\allones}{\mathbf{1}}

\newcommand{\MLE}{\operatorname{MLE}}

\title{Time varying regression with hidden linear dynamics}

\author{Ali Jadbabaie$^\dagger$ \qquad Horia Mania \qquad Devavrat Shah
\qquad Suvrit Sra\\
LIDS, Department of Electrical Engineering and Computer Science \\
 Massachusetts Institute of Technology
}
\date{December 30, 2021}

\newcommand\blfootnote[1]{%
  \begingroup
  \renewcommand\thefootnote{}\footnote{#1}%
  \addtocounter{footnote}{-1}%
  \endgroup
}

\begin{document}
\maketitle

\begin{abstract}
We revisit a model for time-varying linear regression that assumes the unknown parameters evolve according to a linear dynamical system. 
Counterintuitively, we show that when the underlying dynamics are stable the parameters of this model can be estimated from data by combining just two ordinary least squares estimates.
We offer a finite sample guarantee on the estimation error of our method and discuss certain advantages 
it has over Expectation-Maximization (EM), which is the main approach proposed by prior work. 
\end{abstract}

\section{Introduction}

The distribution of labels given the covariates changes over time in a variety of applications of regression. Some example domains where such problems arise include economics, marketing, fashion, and supply chain optimization, where market properties evolve over time. Motivated by such problems, we revisit a model for time-varying linear regression that assumes the unknown parameters evolve according to a linear dynamical system. \blfootnote{$\dagger$ The author names appear in the alphabetical order of their last names. \\\indent \quad\; Corresponding author: Horia Mania, hmania@mit.edu.}

One way to account for distribution change in linear regression is to assume that the unknown model parameters change slowly with time \cite{bamieh2002identification,kalaba1989time,zhang2012inference}. While this assumption simplifies the problem and makes it tractable, it misses on exploiting additional structure available and it also fails to model periodicity (e.g., due to seasonality) present in some problems. As an alternative, we are interested in a dynamic model previously studied by \citet{chow1981econometric}, \citet{carraro1984identification}, and \citet{shumway1988modeling}. 
Given data $\{(x_t, y_t)\}_{t = 0}^{T - 1}$ we assume the label $y_t$ is a linear function of the features $x_t$ with unknown parameters $\beta_t$ that evolve according to linear dynamics: 
\begin{align}
\label{eq:main_problem}
\beta_{t + 1} &= A_\star \beta_t + w_t, \\
y_t &= x_t^\top \beta_t + \epsilon_t,\nonumber
\end{align}
where $w_t$ is \emph{process noise} and $\epsilon_t$ is \emph{observation noise}. The hidden states $\beta_t$ are not observed and the parameter matrix $\Astar$ is unknown. In some applications, such as modeling economic trends, one may be interested in recovering the unknown parameters $\Astar$, which would offer insights into long term trends. When $\Astar$ and the distributions of $w_t$ and $\epsilon_t$ are known prediction is also possible via the Kalman filter. Therefore, our main goal is to estimate $\Astar$ from data $\{(x_t, y_t)\}_{t = 0}^{T - 1}$. 

Estimating $\Astar$ in \eqref{eq:main_problem} is a system identification problem of an unactuated linear system that has a time-varying observation map. Up until now there has been no simple estimator for this problem that has a finite sample guarantee. Counterintuitively, we show that when $\Astar$ is stable it can be estimated from data by combining just two ordinary least squares estimates, one that regresses $y_t^2$ on $x_t x_t^\top$ and one that regresses $y_t y_{t + 1}$ on $x_{t + 1} x_{t }^\top$. When the hidden dynamics are stable these two regressions 
estimate the covariances $\sscov := \EE \beta_t \beta_t^\top$ and $\Astar \sscov = \EE \beta_{t + 1}\beta_t^\top$, from which we can obtain an estimate of $\Astar$. For this reason we call this approach the covariance method (CM).  

In addition to its simplicity, CM also admits a finite sample guarantee on the estimation error. In contrast, the statistical performance of Expectation-Maximization (EM), which is a classical estimation method used by prior work \cite{shumway1982approach}, is not well understood theoretically. We prove the following guarantee for CM. 

\begin{thm*}
Let $\Sigma_w$ be the covariance of $w_t$, $\sigma_\epsilon$ be the standard deviation of $\epsilon_t$, and $\statedim$ be dimension of $x_t$. Suppose $\Astar$ has spectral radius $\rho = \rho(\Astar)$ smaller than one. Then, given a trajectory of length $T$, CM produces with probability at least $1 - \delta$  an estimate $\Ahat$ such that
\begin{align*}
\|\Ahat - \Astar \|_F \leq \frac{1 + \|\Astar\|}{\lambda_{\min}(\sscov)} \sqrt{\frac{\statedim^4}{T\delta}\left(\sigma_\epsilon^4 + \frac{\|\sscov\|^2 }{1 - \rho^2}\right)}.
\end{align*}
\end{thm*}

This results shows that our method estimates $\Astar$ at a $\Ocal(T^{-1/2})$ rate without the need of a good initialization as in the case of EM. Some dependence on $\lambda_{\min}(\sscov)$ is to be expected because when $\sscov$ is rank deficient the states $\beta_t$ do not span all of $\RR^d$, in which case no method would be able to identify $\Astar$ fully. Unsurprisingly, the upper bound on the estimation error blows up as $\rho$ approaches one. The main idea of CM is to estimate $\sscov$, which itself blows up as $\rho$ approaches one. 

We discuss CM in detail and the formal statement of our main result in Section~\ref{sec:cov_est}.
In Section~\ref{sec:meta} we present a meta result that high lights the main idea of our analysis.
Then, in order to ground the comparison between our method and EM in Section~\ref{sec:background} we discuss background on stable linear systems and their identification in the fully observed case, followed in Section~\ref{sec:em} by a discussion of EM. We conclude with a discussion of some open questions. 

\subsection{Notation.} 
Before moving forward we introduce some useful notation. 
We use $y_{s:t}$ to denote the sequence of observations $\{y_s, y_{s + 1}, \ldots, y_{t - 1}\}$, with a similar notation for $x$ and $\beta$. The bold symbols $\vecx$, $\vecy$, and $\vecb$ are used to denote the entire trajectories $x_{0:T}$, $y_{0:T}$, and $\beta_{0:T}$. The norm $\|\cdot \|$ denotes the $\ell_2$ norm for vectors and the operator norm for matrices. We use $\|\cdot \|_F$ to denote the Frobenius norm.
We use $c$ to denote universal constants that may change from line to line.

\subsection{What makes problem (1) challenging?}
\label{sec:challenge}

Classical methods for linear system identification, such as the eigensystem realization algorithm \cite{juang1985eigensystem} and methods based on ordinary least squares \cite{sarkar2021finite,simchowitz2019learning,tsiamis2019finite}, were designed for the identification of time invariant systems of the form:
\begin{align} 
\beta_{t + 1} &= \Astar \beta_t + \Bstar u_t + w_t\\
y_t &= \Cstar \beta_t + \epsilon_t.
\label{eq:classical_problem}
\end{align}
Two factors make the estimation of such systems easier: the time invariance and the actuation of the system through $u_t$. We discuss why either of these two factors make the estimation easier. 

\paragraph{Time-invariant and actuated systems.}
When the dynamics are actuated and time invariant one can regress $y_t$ on a history of inputs $(u_{t - 1}, u_{t - 2}, \ldots)$ in order to estimate the Markov parameters $\Cstar \Astar^j \Bstar$ with $j\in \{0, 1, \ldots, r\}$ for some $r$. Then, one can apply the Ho-Kalman algorithm to extract estimates of $\Astar$, $\Bstar$, and $\Cstar$ from the estimated Markov parameters. \citet{sarkar2021finite} and \citet{simchowitz2019learning} use this high-level approach to offer guarantees on the estimation of systems of the form \eqref{eq:classical_problem}. The success of such an estimation strategy is guaranteed by \citet{oymak2021revisiting}, who bounded the error of the estimates produced by the Ho-Kalman algorithm as a function of the estimation error of the Markov parameters.

\paragraph{Time-varying and actuated systems.}

In our setting the ``$\Cstar$'' matrix (observer matrix) is time-varying which implies time-varying Markov parameters. Since the time-varying observer matrices are assumed known in our setting, if we were able to actuate the dynamical system, we could estimate the Markov parameters $\Astar^j \Bstar$ and then use the Ho-Kalman algorithm to recover the $\Astar$ and $\Bstar$ matrices. However, \eqref{eq:main_problem} is unactuated. 

Moreover, \citet{majji2010time} showed that if one can collect multiple trajectories from the time-varying dynamics by restarting the system, then one can recover all $A_t$, $B_t$, $C_t$. If only one trajectory were available and $A_t$, $B_t$, and $C_t$ were time-varying and unknown, then identification would impossible because we would have only one data point for each time step. 

\paragraph{Time-invariant and unactuated systems.}

This setting corresponds to having $\Bstar = 0$ in \eqref{eq:classical_problem} and can be handled by regressing $y_{t + 1}$ on a history of past observations $y_t, y_{t - 1}, \ldots y_{t - h}$ for some $h$. This approach is successful because one can show that there exists a matrix $M$ such that 
\begin{align*}
y_{t + 1} = M \begin{bmatrix}
y_t \\
y_{t - 1}\\
\vdots\\
y_{t - h}
\end{bmatrix} + \zeta_t,
\end{align*}
for some small $\zeta_t$ \cite{tsiamis2019finite}. However, in the time-varying setting there is no time-invariant matrix $M$ that satisfies this identity. Even if feasible, such an approach would not use the fact that the observer matrices $x_t^\top$ are known. 

\paragraph{Time-varying and unactuated systems.}

Our problem is both time-varying and unactuated and the methods previously discussed cannot be extended easily to address it. As discussed previously, in the time-invariant and unactuated case a simple estimator can solve the problem when both $\Astar$ and $\Cstar$ are unknown. While our problem is time-varying, the observer matrices are given. Hence, only $\Astar$ must be estimated. Our problem has fewer unknown parameters than in the time-invariant case and yet the methods discussed thus far are not applicable.










\section{The covariance method (CM)}
\label{sec:cov_est}

In this section we propose and analyze a method for estimating $\Astar$ that first estimates the covariance of the hidden states $\beta_t$ sampled from the stationary distribution. To estimate the covariance matrix our method applies ordinary least squares to regress $y_t^2$ on $x_t x_t^\top$. Interestingly, the need to estimate the covariance of random vectors that can only be observed through linear measurements arises in the cryo-EM heterogeneity problem \cite{katsevich2015covariance}. We employ the same solution to this covariance estimation problem and offer a finite sample guarantee for this approach.

\paragraph{Assumptions.}In order to estimate the covariance of the stationary distribution we need the stationary distribution to exist and have a finite covariance. The dynamics $\beta_{t + 1} = \Astar \beta_t + w_t$ have a stationary distribution with finite covariance if and only if the matrix $\Astar$ is stable (i.e., $\rho(\Astar) < 1$). We also assume the covariates $x_t$ are sampled i.i.d. according to some distribution. For simplicity we assume $x_t \iid \Ncal(0, I_\statedim)$, but our analysis can be generalized to features with a sub-Gaussian distribution and a more general covariance structure. We also assume $w_t \iid \Ncal(0, \Sigma_w)$  and $\epsilon_t \iid \Ncal(0, \sigma_\epsilon^2)$, but our analysis can be extended to any independent noise processes with bounded fourth moments. In Appendix~\ref{app:features} we present a more general and sufficient set of assumptions for our analysis. 
 
When $\Astar$ is stable, the linear system mixes exponentially fast to a distribution with covariance $\sscov$ (the distribution is Gaussian when $w_t$ are Gaussian), where $\Sigma_\infty$ is the unique positive semidefinite solution $P$ of the Lyapunov equation $P = A P A^\top + \Sigma_w$. Given that the convergence to stationarity is fast we also assume that the initial state $\beta_0$ is sampled from the stationary distribution $\Ncal(0, \Sigma_\infty)$. Hence, the marginal distribution of all $\beta_t$ is $\Ncal(0, \Sigma_\infty)$. 


\paragraph{The method.}Given our assumptions, we observe that there exist random variables $\zeta_t$ and $\xi_t$ whose means are zero conditional on $x_t$ and $x_{t + 1}$ and that satisfy the identities:
\begin{align}
y_t^2 &= x^\top_t \Sigma_\infty x_t + \sigma_\epsilon^2 + \zeta_t  \quad \text{and} \quad
y_t y_{t + 1} = x^\top_{t+ 1} A \Sigma_\infty x_{t} + \xi_t.
\label{eq:cov_problem}
\end{align} 
We write out explicitly $\zeta_t$ and $\xi_t$ in Appendix~\ref{sec:proof_cov_thm}. Since $\Sigma_\infty$ and $A \Sigma_\infty$ enter \eqref{eq:cov_problem} linearly,  we estimate them with ordinary least squares. Concretely, our method performs the following steps\footnote{For the analysis we assume $\sigma_\epsilon^2$ is given, but in practice one would run OLS with an intercept also to account for $\sigma_\epsilon^2$.}:
\begin{align}
\label{eq:cov_method}
\sshat &\in \argmin_{M} \sum_{t = 0}^{T - 1} (y_t^2 - x_t^\top M x_t - \sigma_\epsilon^2)^2,\\\nonumber
\widehat{M} &\in \argmin_M \sum_{t = 0}^{T - 1} (y_t y_{t + 1} - x_t^\top M x_{t + 1})^2,\\
\Ahat &= \widehat{M}\sshat^{-1}.\nonumber
\end{align}

To offer a guarantee on the estimation error $\Ahat - \Astar$ we first analyze the error in the estimation of $\Sigma_\infty$ and $A \Sigma_\infty$. It is known that ordinary least squares is a consistent estimator for the recovery of covariance matrices from squared linear measurements such as \eqref{eq:cov_problem} when the measurements are i.i.d. \cite{katsevich2015covariance}. However, in our case the measurements $y_t$ are not independent. Moreover, we wish to quantify the estimation rate. 
Analyzing the performance of OLS applied to \eqref{eq:cov_problem} is challenging because $\zeta_t$ and $\xi_t$ are heavy tailed, they are dependent on the covariates $x_t$, and they are dependent across time due to the dynamics. By using Chebyshev's inequality we circumvent these issues, yielding a medium probability guarantee.

Before we state our main result we review a consequence of Gelfand's formula for stable matrices. 
For any stable matrix $\Astar$ and any $\gamma \in (\rho(\Astar), 1)$ there exists a finite $\tau$ such that $\|\Astar^k\| \leq \tau \gamma^k$ for all $k \geq 0$. We denote $\tau(\Astar, \gamma) := \sup\{\|\Astar^k\|\gamma^{-k}: k \geq 0\}$ \cite[see][for more details]{tu2017non}.

\begin{theorem}
\label{thm:cov1}
Suppose $x_t \iid \Ncal(0, I_\statedim)$, $w_t \iid \Ncal(0, \Sigma_w)$, and $\epsilon_t \iid \Ncal(0, \sigma_\epsilon^2)$. Moreover, assume $\rho(\Astar) < 1$ and let $\gamma \in (\rho(\Astar), 1)$. Then, with probability $1 - \delta$ the covariance method produces estimates $\sshat$ and $\widehat{M}$ that satisfy
\begin{subequations}
\begin{align}
\|\sshat - \sscov\|_F &\leq  \sqrt{\frac{c\statedim^2}{T\delta}\left[\sigma_\epsilon^4 + \frac{\tau(\Astar, \gamma)^2}{1 - \gamma^2} \|\sscov\|^2 \! \left(\statedim^2 \! + \log\left(\frac{2T}{\delta}\right)^2\right)\right]},
\label{eq:thm_part_I}\\
\|\widehat{M} -  A\sscov \|_F &\leq \sqrt{\frac{c\statedim^2}{T\delta}\left[\sigma_\epsilon^4 + \frac{\tau(\Astar, \gamma)^3}{1 - \gamma^2} \|\sscov\|^2 \! \left(\statedim^2 \! + \log\left(\frac{2T}{\delta}\right)^2\right)\right]},
\label{eq:thm_part_II} 
\end{align}
\end{subequations}
whenever $T \geq c  \statedim^4 \log(c \frac{\outdim^2}{\delta})$. The term $c$ a universal constant that changes from line to line (the only difference between the two upper bounds is the exponent of $\tau$). 
\end{theorem}

The proof of this result is based on the meta-analysis for linear regression with heavy tailed and dependent measurements presented in Section~\ref{sec:meta}. Theorem~\ref{thm:cov1} shows that as the number $T$ of data points collected grows the estimation error decays at a rate $\Ocal(T^{-1/2})$. Theorem~\ref{thm:cov1} also suggests that the closer to instability the matrix $\Astar$ is the more data our method requires to recover $\sscov$ and $\Astar \sscov$. This relationship is unsurprising since $\sscov$ and the variances of $\zeta_t$ and $\xi_t$ in \eqref{eq:cov_problem} approach infinity as $\rho(\Astar)$ approaches one. Nonetheless, in the case of time-invariant linear dynamical systems with partial observations it is possible to estimate $\Astar$ even when it is marginally unstable (i.e. $\rho(\Astar) = 1$) \cite{simchowitz2019learning,tsiamis2019finite}. 

Despite its merits, Theorem~\ref{thm:cov1} is likely not capturing the right dependence of problem \eqref{eq:main_problem} on the state dimension $d$, the probability of failure $\delta$, and the variances  $\sigma_\epsilon^2$ and $\sscov$. For example, we are trying to estimate $\statedim^2$ unknown parameters so we expect to need a number of scalar measurements that scales as $\statedim^2$, not as $\statedim^4$. We obtain a $1/\delta$ dependence on $\delta$ instead of $\log(1/\delta)$ because the proof relies on Chebyshev's inequality. It is not clear whether the sub-optimal dependencies on problem parameters represent an artifact of our analysis or a limitation of our method. 
  
Our sample complexity guarantee on the estimation of $A$ is a simple corollary of Theorem~\ref{thm:cov1}. 

\begin{corollary}
\label{cor:cov}
Under the same assumptions as in Theorem~\ref{thm:cov1}, as long as $T$ satisfies
\begin{align*}
T \geq \frac{c\statedim^2}{\lambda_{\min}(\sscov)^2\delta}\left[\sigma_\epsilon^4 + \frac{\tau(\Astar, \gamma)^3}{1 - \gamma^2} \|\sscov\|^2 \! \left(\statedim^2 \! + 4\log\left(\frac{2T}{\delta}\right)^2\right)\right],
\end{align*} 
the covariance method outputs with probability at least $1 - \delta$ an estimate $\Ahat$ that satisfies:
\begin{align}
\|\Ahat - \Astar \|_F \leq c\frac{1 + \|\Astar\|}{\lambda_{\min}(\sscov)} \sqrt{\frac{c\statedim^2}{T\delta}\left[\sigma_\epsilon^4 + \frac{\tau(\Astar, \gamma)^3}{1 - \gamma^2} \|\sscov\|^2 \! \left(\statedim^2 \! + 4\log\left(\frac{2T}{\delta}\right)^2\right)\right]}. 
\end{align}
\end{corollary}

Therefore, we have shown that it is possible to estimate $\Astar$ at a rate $\Ocal(T^{-1/2})$ without the need of a good initialization as in the case of EM. This corollary inherits all the limitations of Theorem~\ref{thm:cov1}. The dependence on the inverse of $\lambda_{\min}(\sscov)$ may seem sub-optimal, but it does capture an important aspect of the problem. When $\sscov$ is smaller in a PSD sense it means that the states $\beta_t$ have smaller norms, which implies that the signal-to-noise ration in the observations $y_t = x_t^\top \beta_t + \epsilon_t$ is lower. In particular, if $\lambda_{\min}(\sscov) = 0$,  the hidden states $\beta_t$  lie in a subspace of $\RR^\statedim$ of dimension smaller than $\statedim$, in which case $\Astar$ cannot be fully recovered.

\section{A meta-analysis of linear regression with heavy-tailed and dependent noise}
\label{sec:meta}

In this section we offer a meta-analysis for linear regression that encompasses both the estimation of $\sscov$ and $\Astar \sscov$. Suppose we wish to estimate $\theta$ from data $\{(\phi_t, z_t)\}_{t  = 1}^T$, with $z_t = \phi_t^\top \theta + \zeta_t$, where $\phi_t$ are given features and $\zeta_t$ is noise. However, we do not assume $\zeta_t$ and $\phi_t$ are independent nor do we assume the noise terms $\zeta_t$ are independent across time. Moreover, we allow $\zeta_t$ to be heavy tailed. In this case, standard analyses of linear regression do not apply. To surmount these challenges and prove the following lemma we rely on Chebyshev's inequality. 

\begin{lemma}
\label{lem:meta}
Suppose $\phi_t, \theta \in \RR^k$ and that $\sum_{t = 1}^T \phi_t \phi_t^\top$ is invertible. Also, suppose $\EE[\zeta_t | \phi_{1:T}] = 0$, $\EE[\zeta_t^2 | \phi_{1:T}] \leq f_1(\phi_{1: T})$ and $\EE [\zeta_t \zeta_{t + h}| \phi_{1:t}]] \leq \gamma^h f_2(\phi_{1:T})$ for some $\gamma \in (0,1)$ and functions $f_1$, $f_2$. Then, with probability $1- \delta$ the OLS estimate $\that$ satisfies
\begin{align*}
\|\that - \theta\| \leq \sqrt{\frac{k}{\delta} \left(f_1(\phi_{1:T}) + \frac{f_2(\phi_{1:T})}{1 - \gamma}\right)\left(\sum_{t = 1}^T\phi_t \phi_t^\top \right)^{-1}}.
\end{align*}
\end{lemma}
\begin{proof}
The OLS estimate $\that$ is given by $\left(\sum_{t = 0}^{T - 1}\phi_t \phi_t^\top \right)^{-1}\left(\sum_{t = 0}^{T - 1} \phi_t z_t\right)$. Since $z_t = \phi_t^\top \theta + \zeta_t$, the estimation error can be written as 
\begin{align*}
e:= \that - \theta = \left(\sum_{t = 0}^{T - 1}\phi_t \phi_t^\top \right)^{-1}\left(\sum_{t = 0}^{T - 1} \phi_t \zeta_t\right).
\end{align*}
It is convenient to use Chebyshev's inequality to upper bound $\|e\|$ because it only requires an upper bound on the conditional variance of the noise, conditioned on the covariates $\phi_t$. Let  $\Sigma_{e|x} := \EE\left[ e e^\top| \phi_{1:T}\right]$. Then, Chebyshev's inequality applied to $e$ yields  
\begin{align}
\label{eq:cheb_meta}
\PP\left(\sqrt{e^\top \Sigma_{e|x}^{-1}e} > z\right) \leq \frac{k}{z^2}.
\end{align}

Therefore, we wish to upper bound $\Sigma_{e|x}$ with high probability. Note that
\begin{align}
\Sigma_{e|x} &= \left(\sum_{t = 0}^{T - 1} \phi_t \phi_t^\top \right)^{-1} \left(\sum_{i,j = 0}^{T - 1} \EE\left[\zeta_i \zeta_j|\phi_{1:T}\right] \phi_i \phi_j^\top \right) \left(\sum_{t = 0}^{T - 1} \phi_t \phi_t^\top \right)^{-1}
\label{eq:error_cov_id}
\end{align}
To upper bound $\EE\left[\zeta_i \zeta_j|\phi_{1:T}\right] \phi_i \phi_j^\top $ note that for any vectors $u$ and $v$ we have $u v^\top + v u^\top \preceq uu^\top + vv^\top$. Then, using our assumptions on $\EE\left[\zeta_i \zeta_j|\phi_{1:T}\right]$ yields 
 \begin{align}
 \label{eq:cov_upper_meta}
\Sigma_{e|x} &\preceq \left(f_1(\phi_{1:T}) + \frac{f_2(\phi_{1:T})}{1 - \gamma}\right)\left(\sum_{t = 1}^T\phi_t \phi_t^\top \right)^{-1}.
\end{align}
The conclusion follows by putting together \eqref{eq:cheb_meta} and \eqref{eq:cov_upper_meta}. 
\end{proof}

In the estimation of $\sscov$ and $\Astar \sscov$ the features $\phi_t$ are i.i.d. sub-exponential, allowing us to prove $\lambda_{\min}\left(\sum_{t = 1}^T\phi_t \phi_t^\top \right) \geq c T$. Also, for \eqref{eq:cov_problem} we show that $f_1(\phi_{1:T})$ and $f_2(\phi_{1:T})$ grow at most poly-logarithmically with $T$, leading to Theorem~\ref{thm:cov1}.

\section{Linear system identification with known states}
\label{sec:background}

We ground our discussion of CM and EM by looking at the case when the states $\beta_t$ are given. In the fully observed case $\Astar$ can be estimated by the ordinary least squares (OLS) estimate
\begin{align*}
\Ahat = \left(\sum_{t = 0}^{T - 1} \beta_{t + 1} \beta_t^\top \right) \left(\sum_{t = 0}^{T - 1} \beta_t \beta_t^\top \right)^{-1}. 
\end{align*}
In the partially observed case we cannot compute $\beta_{t + 1} \beta_t^\top$ and $\beta_t \beta_t^\top$ directly. Nonetheless, EM and CM get a handle on these quantities in different ways. For example, EM uses the available data and an initial guess of the unknown parameters $\Astar$ to compute conditional expectations of  $\beta_{t + 1} \beta_t^\top$ and $\beta_t \beta_t^\top$. Then, EM uses these conditional expectations to update it's estimate of $\Astar$, repeating these steps until convergence. 
When the dynamics are stable the averages $\frac{1}{T}\sum_{t = 0}^{T - 1} \beta_{t} \beta_t^\top$ and $\frac{1}{T}\sum_{t = 0}^{T - 1} \beta_{t + 1} \beta_t^\top$ converge to $\sscov$ and $\Astar \sscov$  respectively. Despite not having access to the hidden states, CM estimates $\Astar \sscov$ and $\sscov$. 

Recent results have elucidated the sample complexity of estimating $\Astar$ when the states are given. 
Generally, theoretical results show that the sample complexity of estimating a Markov chain increases with the mixing time of the chain. These analyses exploit the fact that data points that occur with sufficient time delay are approximately independent. Then, they claim that the effective number of data points is equal to the total number divided by the mixing time. However, in the case of linear dynamical systems this line of reasoning does not capture the true relationship between sample complexity and mixing time. \citet{simchowitz2018learning} showed that given a trajectory  produced by the dynamics $\beta_{t + 1} = \Astar \beta_t + w_t$ with $\rho(\Astar) \leq 1$ one can use the ordinary least squares estimator
$
\Ahat = \argmin_A \sum_{t = 0}^{T - 1} \|A\beta_t - \beta_{t + 1}\|^2
$
to obtain an estimate $\Ahat$ that, with probability at least $1 - \delta$, has the following guarantee on the estimation error (informally):
\begin{align}
\|\Ahat - \Astar \| \leq c \sqrt{\frac{d\log(T/ \delta)}{T \lambda_{\min}(\sum_{t = 0}^T \Astar (\Astar^\top)^t)}},
\end{align}
where $c$ is a universal constant. This result is counterintuitive because it applies to all marginally unstable systems (i.e., those with $\rho(\Astar) \leq 1$) and, furthermore, it implies that when all of $\Astar$'s eigenvalues have magnitude one, the estimation rate is $\widetilde{\Ocal}\left(\frac{\sqrt{d}}{T}\right)$ instead of the usual $\widetilde{\Ocal}\left(\sqrt{\frac{d}{T}}\right)$. In other words, in some situations, as the mixing time of the system increases, the estimation rate improves. For scalar linear systems, this property has been know since the work of \citet{white1958limiting}. \citet{sarkar2019near} proved such a guarantee for a class of unstable linear systems and \citet{simchowitz2019learning} derived a similar guarantee for partially observable systems. 

Therefore, when the state is directly observable a less stable system is easier to estimate. It remains an open question to determine whether the same phenomenon occurs for \eqref{eq:main_problem}. \citet{douc2011consistency} showed that the maximum likelihood estimate (MLE) is consistent for the estimation of a class of Markov chains that contains stable partially observed linear systems. However, to the best of knowledge, proving that the MLE is consistent for the estimation of systems that are linear, marginally unstable, and partially observed is an open question.


\section{Expectation-Maximization (EM)}
\label{sec:em}

The classical expectation-maximization method (EM) that has already been applied to \eqref{eq:main_problem} \cite{shumway1982approach}. While EM performs well for certain instances of problem ~\eqref{eq:main_problem}, it has several drawbacks that motivated us to search for a different estimator. 

EM is difficult to analyze theoretically and it can converge to a local maximizer of the likelihood when it is not initialized close enough to a global maximizer, leading to a large estimation error \cite{balakrishnan2017statistical}. Furthermore, in our experience a straightforward implementation of EM is numerically unstable when applied to \eqref{eq:main_problem} with $x_t \iid \Ncal(0, I_\statedim)$, a problem that does not occur with a constant $x_t$. Finally, EM needs to know something about the distributions of the process and observation noise processes. Classically, the noise processes are assumed Gaussian.

EM approximately maximizes the likelihood of a probability model that depends on unobserved latent variables. In fact, EM was one of the first estimation methods used to tackle problem \eqref{eq:main_problem} \cite{shumway1982approach}. 
The maximum likelihood estimate $A^{\MLE}$, as the name suggests, is the maximizer of the likelihood $\PP_A \left(\vecx, \vecy\right)$. The likelihood is a nonconvex function of $A$ and it is not known how to provably find $A^{\MLE}$. EM is a local alternating maximization algorithm that is guaranteed to converge to a local maximizer of the likelihood. 

EM starts with an initial guess of $\Astar$ and alternates between two steps. In the E-step it constructs a lower bound on the likelihood of the data $\{(x_t, y_t)\}_{t = 0}^{T - 1}$ by computing the conditional covariances $S_{t} = \EE[\beta_{t } \beta_t^\top | \vecx, \vecy, A^\prime]$ and $S_{t, t - 1} = \EE[\beta_{t + 1} \beta_t^\top | \vecx, \vecy, A^\prime]$. To compute $S_{t}$ and $S_{t, t - 1}$ EM uses the Kalman filter to compute the conditional expectations with respect to past data: $\EE[\beta_t | x_{0:t}, y_{0:t}, A^\prime]$, $\EE[\beta_t \beta_t^\top | x_{0:t}, y_{0:t}, A^\prime]$, and $\EE[\beta_{t + 1} \beta_t^\top| x_{0:t}, y_{0:t}, A^\prime]$, followed by a backwards pass over the data to smooth these conditional expectations. In the M-step, EM simply updates the current guess of $\Astar$ to $\left(\sum_{t = 1}^{T - 1} S_{t, t - 1}\right)\left(\sum_{t= 1}^{T - 1} S_{t - 1}\right)^{-1}$. The detailed steps are shown by \citet{shumway2000time}. 

Analyses of EM are difficult even when presented with i.i.d. data. While one could follow the analysis of \citet{balakrishnan2017statistical} to guarantee that EM converges to $A^{\MLE}$ if initialized close enough, we believe this approach would lead to a result that suggests EM performs poorly when $\Astar$ 
is close to being unstable. However, given what we know from the performance of the maximum likelihood estimate (the OLS estimate) in the fully observed case, it is likely that EM would perform better when $\Astar$ has eigenvalues close to the unit circle. 
In fact, EM operates by computing the best guess of the hidden states $\beta_t$ given the observed data. Therefore, it is likely that less stable systems offer an additional benefit to EM. When $\Astar$ has larger eigenvalues the dynamics have longer memory, which allows for a better estimation of the hidden states. For example, if $\Astar$ was the zero matrix, then the dynamics would have no memory and the observations $y_{t + 1}$ and $y_{t - 1}$ would not hold any useful information for inferring $\beta_t$. 


In contrast, the performance of CM degrades as $\rho(\Astar)$ approaches one. Nonetheless, our method does not depend on a good initialization, does not require knowledge concerning the noise processes, and admits a simple theoretical guarantee.

\section{Related work}

Linear dynamical systems have been studied in depth as models for time-series data and we cannot do justice to its rich history or to the broader work on linear system identification. \citet{brockwell2009time} and \citet{ljung1987system} discuss these topics in detail. We focus on the closest and most recent related works. 

\cite{chow1981econometric} first proposed \eqref{eq:main_problem} as a model for time-varying regression, with \citet{carraro1984identification} later generalizing their analysis. The method studied by \citet{carraro1984identification} resembles in some ways are own. As an intermediate step it estimates the covariances between different innovations of a Kalman filter, but compared to CM it is complicated and no finite-sample guarantees have been provided. 

\citet{shumway1982approach} developed model \eqref{eq:main_problem} from a different perspective and proposed EM as a viable estimator. They were interested in modelling time-series with missing data and in their model $x_t$ is a matrix of zeros and ones that encodes which entries are available and which are missing in a set of time-series. Then, \citet{khan2007expectation} applied EM to fit a slightly generalized model \eqref{eq:main_problem} on EEG time series; they assumed the hidden states $\beta_t$ follow an auto-regressive model. 

\citet{lubik2015time} noted that macroeconomic time-series often display nonlinear behaviors that can be captured with time-varying linear models. In particular, they studied a time-varying auto-regressive model in which the unknown parameters evolve randomly and proposed a Bayesian inference method. In a similar vein, \citet{isaksson1987identification} studied actuated linear dynamics with randomly evolving parameters and proposed an adaptive Kalman filtering approach for identification. Most prior work on the identification of time-varying linear systems considered the case in which multiple experiments can be conducted on the system whose parameters undergo the same variation \cite{liu1997identification,majji2010time}. 

The recent development of non-asymptotic guarantees in statistics has led to similar results in system identification. In addition to the works previously mentioned in Sections \ref{sec:challenge} and \ref{sec:background}, we mention a few other such works. \cite{hardt2016gradient} showed that a stochastic gradient descent can recover the parameters of certain linear systems in polynomial time. \cite{sarkar2021finite} proved a non-asymptotic guarantee for the identification of partially observed linear dynamics of unknown order. \cite{hazan2017learning} and \cite{hazan2018spectral} proposed spectral filtering methods that predict with sublinear regret the next output of unknown partially observed linear systems. When actuation is available \citet{wagenmaker2020active} showed that active learning can be used for a faster identification of linear dynamics. Finally, \citet{tsiamis2020sample} analyzed the sample complexity of designing  a good Kalman filter based on observed data from an unknown linear systems.




\section{Conclusion and open questions}{}

Counterintuitively, we have shown that it is possible to fit a model for time-varying linear regression by combining just two ordinary least squares estimates. In contrast to EM, our method does not require a good initialization nor knowledge regarding the noise processes. Moreover, CM admits a simple theoretical guarantee that shows the estimation error decays at a $\Ocal(1/\sqrt{T})$ rate.

CM has its own drawbacks. Firstly, it is applicable only to stable systems and its performance degrades as $\rho(\Astar)$ approaches one. Based on the statistical rate achievable when the  states $\beta_t$ are directly observable it should be possible to have a method that performs better when $\Astar$ is less stable. Secondly, we believe our guarantee for CM depends suboptimally on the failure probability $\delta$, the dimension $\statedim$, and the variances $\sigma_\epsilon^2$ and $\Sigma_w$. We leave this issue and several other open questions for future work:
\begin{itemize}
\item Recent work in robust statistics have led to new algorithms and guarantees for linear regression with heavy-tailed noise \citet{cherapanamjeri2020algorithms,depersin2020spectral}. Can one extend those results to our setting? These methods may address the suboptimal dependence of CM on some of the quantities previously mentioned. 

\item Similarly to the fully observed setting \cite{simchowitz2018learning}, can one guarantee that the performance of the  maximum-likelihood estimator improves as $\Astar$ becomes less stable? 

\item Can one characterize the statistical rate of EM? Determining how closely to the MLE one needs to initialize EM would be of particular interest. 

\item What can one say when the hidden states $\beta_t$ evolve according to more general linear dynamics or even nonlinear dynamics? 
\end{itemize}

\bibliographystyle{abbrvnat}   
\bibliography{dynamiclr} 

\appendix


\section{Proof of bound~\ref{eq:thm_part_I} of Theorem~\ref{thm:cov1}}
\label{sec:proof_cov_thm}

In this section we analyze formally the performance of the method relying on covariance estimation.

The particular structure of the noise terms $\zeta_t$ and $\xi_t$ appearing in the squared observation equations \eqref{eq:cov_problem} is important. If we denote by $Z_t$ the zero-mean random matrix $\beta_t\beta_t^\top - \Sigma_\infty$, we have 
\begin{align}
\zeta_t &= x^\top_t Z_t x_t + 2 \epsilon_t x_t^\top \beta_t + (\epsilon^2 - \sigma_\epsilon^2), \nonumber\\
\xi_t &= x^\top_{t + 1} A  Z_t  x_{t} +\epsilon_{t + 1} x_t^\top \beta_t + x_t^\top \beta_t w_t^\top x_{t + 1}  + \epsilon_t  x_{t + 1}^\top A \beta_t + \epsilon_t \epsilon_{t + 1} + \epsilon_t w_t^\top x_{t + 1}.
\label{eq:noise_cov}
\end{align} 
Therefore, since $\epsilon_t$ and $\beta_0$ are independent and have zero means, $\zeta_t$ and $\xi_t$ are zero-mean random variables even when we condition on $x_t$ and $x_{t + 1}$. 

Before turning to the analysis we introduce some useful notation. We use $\svec \colon \Sym_\outdim \to \RR^{\outdim(\outdim + 1) / 2}$ to denote a flattening of symmetric matrices that does not repeat the off-diagonal entries twice, but scales them by $\sqrt{2}$. This choice ensures that for any symmetric matrix $M$ we have $\|M\|_F = \|\svec(M)\|_2$. We also denote $\phi_t := \svec(x_t x_t^\top)$. Below we prove that $\sum_{t = 0}^{T - 1} \phi_t \phi_t^\top$ is invertible with high probability when $T$ is sufficiently large. Then, since the variance $\sigma_\epsilon^2$ of the observation noise is assumed known, we have
\begin{align}
\svec(\sshat) &= \left(\sum_{t = 0}^{T - 1} \phi_t \phi_t^\top \right)^{-1} \left(\sum_{t = 0}^{T - 1}\phi_t (y_t^2 - \sigma_\epsilon^2)\right)\text{, and} \nonumber \\
\svec(\sshat) - \svec(\sscov) &= \left(\sum_{t = 0}^{T - 1} \phi_t \phi_t^\top \right)^{-1} \left(\sum_{t = 0}^{T - 1}\phi_t \zeta_t\right) \label{eq:cov_est_error}
\end{align}
Similar identities hold for the estimate $\widehat{M}$ of $\Astar \sscov$. Since the matrices $\Astar \sscov$ and $x_t x_{t + 1}^\top$ are not symmetric in general, for the analysis of this component we work with a conventional flattening operator $\vec \colon \Sym_\statedim \to \RR^{\statedim^2}$ instead. We return to the proof of \eqref{eq:thm_part_II} in Appendix~\ref{sec:thm_partII}.

To prove that \eqref{eq:cov_est_error} note that all the assumptions of Lemma~\ref{lem:meta} are satisfied. We are left to lower bound the minimum eigenvalue of the design matrix $\sum \phi_t \phi_t^\top$ and to upper bound the correlations of the noise terms $\zeta_t$. We first lower bound the minimum eigenvalue. 

\subsection{Lower bounding the design matrix}
\label{sec:design_matrix}

We wish to show that $\lambda_{\min} \left(\sum_{t = 0}^{T - 1} \phi_t \phi_t^\top  \right) = \Omega(T)$ with high probability. The first thing we observe is that $\EE \phi_t \phi_t^\top$ is non-degenerate:
\begin{align}
\Sigma_\phi := \EE \phi_t \phi_t^\top = 
\begin{bmatrix}
2 I_{\statedim (\statedim - 1)/2} & 0 \\
0 & 2 I_{\statedim} + \allones \allones^\top
\label{eq:cov_phi}
\end{bmatrix},
\end{align}
which is a matrix with the smallest eigenvalue equal to $2$. Now, to lower bound the minimum eigenvalue of $\sum_t \phi_t \phi_t^\top$
we rely on a standard covering argument. Namely, we first show that for any fixed unit vector $v\in \RR^{\statedim(\statedim + 1)/ 2}$  we have $\sum_t (v^\top \phi_t)^2 = \Omega(T)$. Then, we choose an appropriately fine covering of $\Sbb^{\statedim(\statedim + 1)/ 2 - 1}$ and apply the union bound. The main technical challenge with this approach is that the resolution of the cover has to be chosen in terms of an upper bound on the maximum eigenvalue of $\sum_t \phi_t \phi_t^\top$. Following this argument we can offer the following guarantee. 

\begin{prop}
Let $x_t \iid \Ncal(0, I_\outdim)$ and $\phi_t = \svec(x_t x_t^\top)$. Also, let $m_\phi = \sup_{\|v\| = 1} \EE (v^\top \phi_t)^4$. Then, there exists a positive constant $c$ so that as long as $T \geq c \outdim^4 \log(c \frac{\outdim^2}{\delta})$ we have 
\begin{align}
\PP\left(\lambda_{\min} \left(\sum_{t = 0}^{T - 1} \phi_t \phi_t^\top \right) > T\right) \geq 1 - \delta.
\end{align}
\label{prop:guarantee_on_exploration}
\end{prop}

This result suggests that $\statedim^4$ samples are needed before $\sum_t \phi_t \phi_t^\top$ is invertible with high probability. Despite this burn in period, Proposition~\ref{prop:guarantee_on_exploration} shows that $\lambda_{\min}\left( \sum_{t = 0}^{T - 1} \phi_t \phi_t^\top \right) = \Omega(T)$.  


To prove Proposition~\ref{prop:guarantee_on_exploration} we start by lower bounding $\sum_t (v^\top \phi_t)^2$ for a fixed unit vector $v$. 

\begin{lemma}
\label{lem:lower_bound_cov_single_v}
Let $x_t \iid \Ncal(0, I_\outdim)$ and $\phi_t = \svec(x_t x_t^\top)$.
Then, for a given $v\in\Sbb^{\outdim(\outdim + 1)/ 2 - 1}$, we have 
\begin{align}
\PP\left( \sum_{t = 0}^{T - 1} (v^\top \phi_t)^2 \geq 2 T - \sqrt{c T \statedim^2 \log(1/\delta)}\right) \geq 1 - \delta. 
\end{align}

\end{lemma}
\begin{proof}
Since the covariance of $\phi_t$ is equal to \eqref{eq:cov_phi}, we have $\EE (v^\top \phi_t)^2  = \EE v^\top \phi_t \phi_t^\top v = v^\top \Sigma_\phi v \geq 2$. Then, applying one-sided Bernstein's  inequality for non-negative random variables  \cite[see][Chapter 2]{wainwright2019high}, we find that
\begin{align}
\PP\left( \frac{1}{T}\sum_{t = 0}^{T - 1} (v^\top \phi_t)^2 \leq 2 - z \right) \leq \exp\left(-\frac{T z^2}{2 \EE[(v^\top \phi_t)^4]}\right). 
\end{align}
Let us denote $m_\phi := \sup_{\|v\| = 1} \EE (v^\top \phi_t)^4$.  Then, we can choose $z = \sqrt{\frac{2 m_\phi \log(1/\delta)}{T}}$. 

To complete the proof we need to upper bound $m_\phi$. Note that for any unit vector $v$ there exists a symmetric matrix $V$ with $\|V\|_F \leq 2$ such that $v^\top \phi = x^\top V x$, where $\phi = \svec(xx^\top)$. Hence, we must upper bound $\EE (x^\top V x)^4$ with $x \sim \Ncal(0, I_\statedim)$. According to \citet{magnus1979expectation} we have 
\begin{align*}
\EE (x^\top V x)^4 = \tr(V)^4 + 12 \tr(V)^2 \tr(V^2) + 32\tr(V)\tr(V^3) + 12 \tr(V^2)^2 + 48 \tr(V^4).
\end{align*}
Since $\|V\|_F \leq 2$, we can easily find that $\tr(V)^4 \leq 16 \statedim^2$ with the other terms having a lower order dependence on $\statedim$. Hence, $\EE (x^\top V x)^4 = \Ocal(\statedim^2)$. 
\end{proof}

The next lemma is useful in controlling the size of the maximum eigenvalue of $S_T = \sum_{t = 0}^{T - 1} \phi_t \phi_t^\top$, which is needed in the covering argument. 

\begin{lemma}
\label{lem:upper_bound_cov}
If $x_t \iid \Ncal(0, I_\outdim)$ and $\phi_t = \svec(x_t x_t^\top)$, we have 
\begin{align}
\PP\left(\lambda_{\max} \left(\sum_{t = 0}^{T - 1} \phi_t \phi_t^\top \right) < \frac{4T \outdim^2}{\delta}\right) \geq 1 - \delta
\end{align}
\end{lemma}
\begin{proof}
Markov's inequality implies 
\begin{align}
\PP\left(\lambda_{\max} \left(S_T \right) > \frac{4T \outdim^2}{\delta}\right) &\leq \frac{\delta}{4T \outdim^2} \EE\left[\lambda_{\max} \left(S_T \right)\right]\\
&\leq \frac{\delta}{4 T \outdim^2} \EE\left[ \tr\left(S_T \right)\right] = \frac{\delta}{4 \outdim^2} \tr\left(\Sigma_\phi\right). 
\end{align}
From \eqref{eq:cov_phi} we know that $\tr(\Sigma_\phi) \leq 4d^2$, which gives the desired conclusion. 
\end{proof}

Let us return to the proof of Proposition~\ref{prop:guarantee_on_exploration}. Let $v_1, v_2, \ldots, v_{N_\epsilon}$ be a cover of $\Sbb^{\outdim(\outdim + 1) /2 - 1}$ such that for any $v\in \Sbb^{\outdim(\outdim + 1) /2 - 1}$ there exists $i$ such that $\|v - v_i\| \leq \epsilon$. It is well known that one can find $N_\epsilon \leq (2/ \epsilon + 1)^{\outdim (\outdim + 1)/ 2 - 1}$ vectors $v_i$ that satisfy this property. 

For a given unit vector $v$ choose the vector $v_i$ in the cover such that $\|v - v_i\| \leq \epsilon$ and denote $\Delta_i = v - v_i$. Then, we have
\begin{align*}
v^\top S_T v &= (v_i + \Delta_i)^\top S_T (v_i + \Delta_i) = v_i^\top S_T v_i + 2\Delta_i^\top S_T v_i + \Delta_i^\top S_T \Delta_i\\
&\geq v_i^\top S_T v_i  - 2\epsilon \|S_T\|,
\end{align*}
where the last inequality follows because $\Delta_i^\top S_T \Delta_i \geq 0$ and $\|v_i\| = 1$. Since this argument holds for any unit vector $v$, we can take infimum of both sides to find 
\begin{align}
\label{eq:cover_min_eig}
\inf_{\|v\| = 1} v^\top S_T v &\geq \inf_i v_i^\top S_T v_i  - 2\epsilon \|S_T\|. 
\end{align}
Let us denote $\Ecal = \{\lambda_{\max} \left(S_T \right) < \frac{8 T \outdim^2}{\delta}\}$. Then, we can write 

\begin{align*}
\PP\left(\lambda_{\min} \left(S_T \right) < T\right) \leq \PP\left(\left\{\lambda_{\min} \left(S_T \right) < \T \right\}\cap \Ecal \right) + \PP(\Ecal^c). 
\end{align*}
Lemma~\ref{lem:upper_bound_cov} guarantees that $\PP(\Ecal^c) \leq \delta / 2$. 
Now, since $\lambda_{\min}(S_T) = \inf_{\|v\| = 1} v^\top S_T v$, we can apply \eqref{eq:cover_min_eig} to find
\begin{align*}
\PP\left(\left\{\lambda_{\min} \left(S_T \right) < T\right\} \cap \Ecal \right) &\leq \PP\left( \left\{\inf_i v_i^\top S_T v_i  - 2\epsilon \|S_T\| < T\right\}\cap \Ecal \right) \\
&\leq \PP\left( \inf_i v_i^\top S_T v_i  - 16\epsilon \frac{T \outdim^2}{\delta} < T \right)
\end{align*}
We choose $\epsilon = \delta/ (32 \outdim^2)$. Now, we union bound over $i$ and use Lemma~\ref{lem:lower_bound_cov_single_v} with $z = 1$ to find 
\begin{align}
\label{eq:union_bound_min_eig}
\PP\left(\left\{\lambda_{\min} \left(S_T \right) < T \right\} \cap \Ecal \right) \leq N_\epsilon \exp \left(-\frac{T}{8m_\phi}\right). 
\end{align}
Recall that $N_\epsilon \leq (2/\epsilon + 1)^{\outdim^2}$. Therefore, when 
\begin{align*}
T \geq 8m_\phi \log\left(2\frac{(64 \outdim^2 / \delta + 1)^{\outdim^2}}{\delta}\right) 
\end{align*}
the right hand side of \eqref{eq:union_bound_min_eig} is at most $\delta / 2$. Proposition~\ref{prop:guarantee_on_exploration} follows after some simple computations. 

\subsection{Bounding the conditional variance}
\label{sec:cov_ind_partI}

In this section, in order to upper bound $\Sigma_{e|x}$, we upper bound $\EE[\zeta_t^2 | x_{0:T}]$ and the magnitude of $\EE[\zeta_i \zeta_j | x_{0:T}]$. Recall that 
$\zeta_t = \phi_t^\top \psi_t + 2 \epsilon_t x_t^\top \beta_t + (\epsilon_t^2 - \sigma_\epsilon^2)$, where $\psi_t = \svec(\beta_t \beta_t^\top - \sscov)$ 
Since the cross terms arising from squaring $\zeta_t$ have mean zero, we have
\begin{align}
\EE[\zeta_t^2 | x_{t}] &= \EE \left[\left.(x^\top_t Z_t x_t)^2 + 2 \epsilon_t^2 x_t^\top \beta_t \beta_t^\top x_t + (\epsilon_t^2 - \sigma_\epsilon^2)^2 \right| x_t\right] \nonumber\\
&= 2(x_t^\top \sscov x_t)^2 + 2 \sigma_\epsilon^2 x_t^\top \sscov x_t + 2 \sigma_\epsilon^4\nonumber \\
&\leq 3(x_t^\top \sscov x_t)^2 + 3 \sigma_\epsilon^4.
\label{eq:zeta_bound}
\end{align}
Now we turn to $\EE[\zeta_t \zeta_{t^\prime} | x_{0:T}]$ with $t <  t^\prime$. Given the definition of $\zeta_t$ and since $\epsilon_t$ and $\epsilon^2 - \sigma_\epsilon^2$ are zero mean and independent of $\beta_t$ for all $t$, we find that 
\begin{align}
\label{eq:cross_term_first_decomp}
\EE[\zeta_t \zeta_{t^\prime} | x_{0:T}] &= \EE[ x_t^\top Z_t x_t x_{t^\prime}^\top Z_{t^\prime} x_{t^\prime}| x_{0:T}] = \EE[ \beta_t^\top x_t x_t^\top \beta_t  \beta_{t^\prime}^\top x_{t^\prime} x_{t^\prime}^\top \beta_{t^\prime}| x_{0:T}] - x_t^\top \sscov x_t x_{t^\prime}^\top \sscov x_{t^\prime}
\end{align}
Since $\beta_{t + 1} = \Astar \beta_t + w_t$, we know that 
\begin{align}
\beta_{t^\prime} = \Astar^{t^\prime - t}\beta_t + \underbrace{\sum_{j = 1}^{t^\prime - t} \Astar^{t^\prime - t - j} w_{j - 1 + t}}_{u_{t,t^\prime}}. 
\end{align}
The term $u_{t, t^\prime}$ has mean zero and is independent of $\beta_t$. Therefore, 
\begin{align}
\EE[ \beta_t^\top x_t x_t^\top \beta_t  \beta_{t^\prime}^\top x_{t^\prime} x_{t^\prime}^\top \beta_{t^\prime}| x_{0:T}] &= 
\EE[ \beta_t^\top x_t x_t^\top \beta_t  \beta_{t}^\top (A^{t^\prime - t})^\top x_{t^\prime} x_{t^\prime}^\top A^{t^\prime - t} \beta_{t}| x_{0:T}]\nonumber \\
&+ \EE\left[ \beta_t^\top x_t x_t^\top \beta_t  u_{t, t^\prime}^\top x_{t^\prime} x_{t^\prime}^\top u_{t, t^\prime}| x_{0:T}\right]
\label{eq:cross_term_decomposition}
\end{align}

Given that $\beta_t \sim \Ncal(0, \sscov)$ and the results of \citet{magnus1979expectation}, the first term can be seen to be 
\begin{align}
x_t^\top \sscov x_t x_{t^\prime}^\top \Astar^{t^\prime - t} \sscov (\Astar^{t^\prime - t} )^\top x_{t^\prime} + 2 (x_t^\top \sscov (\Astar^{t^\prime - t} )^\top x_{t^\prime})^2
\end{align}
The second term of \eqref{eq:cross_term_decomposition} can be more easily computed because $\beta_t^\top x_t x_t^\top \beta_t $ and $u_{t, t^\prime}^\top x_{t^\prime} x_{t^\prime}^\top u_{t, t^\prime}$ are independent given $x_{0:T}$. Hence, the second term of \eqref{eq:cross_term_decomposition} is equal to 
\begin{align}
\label{eq:cross_term_second_term}
x_t^\top \sscov x_t x_{t^\prime}^\top \sum_{j = 0}^{t^\prime - t - 1} \Astar^j \Sigma_w (\Astar^j)^\top x_{t^\prime}. 
\end{align}
Recall that $\sscov = \sum_{j = 0}^{\infty} \Astar^j \Sigma_w (\Astar^j)^\top$ and hence $\sum_{j = t^\prime - t}^{\infty} \Astar^j \Sigma_w (\Astar^j)^\top  = \Astar^{t^\prime - t} \sscov (\Astar^{t^\prime - t})^\top$. Then, putting together \eqref{eq:cross_term_first_decomp}, \eqref{eq:cross_term_decomposition}, and \eqref{eq:cross_term_second_term} we find that 
\begin{align}
\EE[\zeta_t \zeta_{t^\prime} | x_{0:T}] &= 2 (x_t^\top \sscov (\Astar^{t^\prime - t} )^\top x_{t^\prime})^2
\end{align}

From the concentration Lipschitz functions of standard Gaussian random vectors we know that $\sup_{t = 0}\|x_t\| \leq \sqrt{d} + \sqrt{2\log(2T/\delta)}$ with probability at least $1 - \delta/2$. We can use this result to upper bound $\EE[\zeta_t^2 | x_{0:T}]$ and $\EE[\zeta_t \zeta_{t^\prime} | x_{0:T}]$. Note that we do this by upper bounding terms of the form $x_t^\top \sscov x_t$ by $\|\sscov\| \|x_t\|^2$. This upper bound is sub-optimal in general since $\EE x_t^\top \sscov x_t = \tr\left(\sscov\right)$. One could try to get final upper bounds in terms of $\tr(\sscov)$ instead of $\statedim \|\sscov\|$ by using the Hanson-Wright inequality. We do not take this approach because of the time varying matrices $\Astar^{t^\prime - t}$.

Recall that for any $\gamma \in (\rho(\Astar), 1)$ we can upper bound $\|\Astar^t\| \leq \tau(\Astar, \gamma)\gamma^t$ for all $t\geq 0$. Therefore, with probability $1 - \delta / 2$ we can upper bound
\begin{align*}
\EE[\zeta^2_t | x_{0:T}] &\leq 3\sigma_\epsilon^4 + 24 \|\sscov \|^2 \left(d^2 + 4\log\left(\frac{2T}{\delta}\right)^2\right),\\
\EE[\zeta_t \zeta_{t^\prime} | x_{0:T}] &\leq 8\|\sscov \|^2 \tau(\Astar, \gamma)^2 \gamma^{2(t^\prime - t)} \left(d^2 + 4\log\left(\frac{2T}{\delta}\right)^2\right).
\end{align*}
Now, we can bound the conditional covariance matrix $\Sigma_{e|x}$. We use identity \eqref{eq:error_cov_id} and the inequality $uv^\top + vu^\top \preceq uu^\top + vv^\top$ that holds for any vectors $u$ and $v$ of equal dimension. Then, with probability at least $1 - \delta / 2$ we have
\begin{align}
\Sigma_{e|x} \preceq \left[3\sigma_\epsilon^4 + 24\frac{\tau(\Astar, \gamma)^2}{1 - \gamma^2} \|\sscov\|^2 \left(d^2 + 4\log\left(\frac{2T}{\delta}\right)^2\right)\right] \left(\sum_{t = 0}^{T - 1} \phi_t \phi_t^\top \right)^{-1}.
\end{align}
Proposition~\ref{prop:guarantee_on_exploration} guarantees that for all $T \geq c \statedim^4 \log(c\statedim^2 / \delta)$ the following bound holds with probability at least $1 - \delta$:
\begin{align}
\label{eq:final_bound_cov}
\Sigma_{e|x} \preceq \frac{c}{T}\left[\sigma_\epsilon^4 + \frac{\tau(\Astar, \gamma)^2}{1 - \gamma^2} \|\sscov\|^2 \left(d^2 + 4\log\left(\frac{2T}{\delta}\right)^2\right)\right].
\end{align}
\subsection{Putting it all together}

Recall that for all $z > 0$ Chebyshev's inequality guarantees that $\PP\left(\sqrt{e^\top \Sigma_{e|x} e} > z\right) \leq \frac{\statedim^2}{z^2}$. 
Let $\Ecal_\sigma$ be the event on which \eqref{eq:final_bound_cov} holds and let
\begin{align}
b := \sqrt{\frac{c\statedim^2}{T\delta}\left[\sigma_\epsilon^4 + \frac{\tau(\Astar, \gamma)^2}{1 - \gamma^2} \|\sscov\|^2 \left(d^2 + 4\log\left(\frac{2T}{\delta}\right)^2\right)\right]}
\end{align}

Therefore, we have 
\begin{align*}
\PP\left(\|e\|_2 > b\right) &\leq \PP\left(\{\|e\|_2 > b\} \cap \Ecal_\sigma \right) + \PP(\Ecal_\sigma^c)\\
&\leq \PP\left(\left\{e^\top \Sigma_{e|x}^{-1}e > \frac{2 \statedim^2}{\delta}\right\} \cap \Ecal_\sigma \right) + \frac{\delta}{2}\\
&\leq \PP\left(e^\top \Sigma_{e|x}^{-1}e > \frac{2 \statedim^2}{\delta}\right) + \frac{\delta}{2} \leq \delta
\end{align*}
This step completes the proof of the first part of Theorem~\ref{thm:cov1}.

\section{Proof of bound \ref{eq:thm_part_II} of Theorem~\ref{thm:cov1}}
\label{sec:thm_partII}

In this section we prove the guarantee offered in Theorem~\ref{thm:cov1}
for the estimation of the product $A \sscov$ from the data points 
$(x_{t + 1} x_t^\top, y_{t + 1} y_{t})$. 

We use the vectorization: $\vec := \RR^{\outdim \times \outdim} \to \RR^{\outdim^2}$. We we follow an analysis similar to that \eqref{eq:thm_part_I}, relying on Lemma~\ref{lem:meta}. Therefore, we define $\phi_t = \vec(x_{t + 1} x_{t}^\top)$ and denote $\widehat{M}$ as the OLS estimate of $A \sscov$. Then,  
\begin{align}
\vec(\widehat{M}) - \vec(A \sscov) = \left(\sum_{t = 0}^{T - 1} \phi_t \phi_t^\top \right)^{-1} \left(\sum_{t = 0}^{T - 1} \phi_t \xi_t\right),
\end{align}
where $\xi_t$ is the noise term shown in \eqref{eq:noise_cov}.

Note that the covariance of $\vec(x_t x_{t + 1}^\top)$ is the identity matrix, whose minimum eigenvalue is one. Therefore, one could follow the same steps as in Section~\ref{sec:design_matrix} to show that $\lambda_{\min} \left(\sum_t \phi_t \phi_t^\top\right) = \Omega(T)$ when $\phi_t = \vec(x_t x_{t + 1}^\top)$. 

We are left to upper bound the conditional correlations of the noise terms $\xi_t$. We follow the same steps as in Section~\ref{sec:cov_ind_partI}; we first upper bound $\EE[\xi_t^2 | x_{0:T}].$
Conditional on $x_t$ and $x_{t + 1}$, all cross terms of $\xi_t^2$ are zero-mean. Hence, we have 

\begin{align*}
\EE[\xi_t^2 | x_{0:T}] &= \EE[ (x^\top_t Z_t  A^\top x_{t + 1})^2 | x] +\sigma_\epsilon^2 x_t^\top \sscov x_t + x_t^\top \sscov x_t x_{t + 1}^\top \Sigma_w x_{t + 1} + \sigma_\epsilon^2  x_{t + 1}^\top A \sscov A^\top x_{t + 1}\\
&  + \sigma_\epsilon^4 + \sigma_\epsilon^2  x_{t + 1}^\top \Sigma_w x_{t + 1}.
\end{align*}

To compute the first term in this expansion we rely on known formulas for the expectation of products of quadratic forms of Gaussian random vectors \cite{bao2010expectation,magnus1979expectation}. Recall that $Z_t = \beta_t\beta_t^\top - \sscov$ with $\beta \sim \Ncal(0, \sscov)$. Then, we obtain
\begin{align*}
\EE[ (x^\top_t Z_t  A^\top x_{t + 1})^2 | x] &=  x_{t + 1} A\sscov A^\top x_{t + 1} x_t^\top \sscov x_t + x_t^\top \sscov  A^\top x_{t + 1} x_{t + 1}^\top A \sscov x_t\\
&\leq  \left(x_{t + 1} A\sscov A^\top x_{t + 1}\right)^2 + \left( x_t^\top \sscov x_t \right)^2
\end{align*}

Therefore, after some simple inequalities we find 
\begin{align*}
\EE[\xi_t^2 | x_{0:T}] &\leq \frac{5}{2}\sigma_\epsilon^4 + 2 \left(x_t^\top \sscov x_t \right)^2 + \frac{3}{2} \left(x_{t + 1} A\sscov A^\top x_{t + 1}\right)^2 + (x_{t + 1}^\top \Sigma_w x_{t + 1})^2. 
\end{align*}
Since $\sscov$ satisfies the Lyaunov equation $\sscov = A \sscov A^\top + \Sigma_w$, we know that $\sscov \succeq \Sigma_w$ and $\sscov \succeq A\sscov A^\top$. Hence, we can further upper bound the conditional variance of the noise by 
\begin{align*}
\EE[\xi_t^2 | x_{0:T}] &\leq \frac{5}{2}\sigma_\epsilon^4 + 2 \left(x_t^\top \sscov x_t \right)^2 + \frac{5}{2} \left(x_{t + 1} \sscov x_{t + 1}\right)^2 \leq 3\sigma_\epsilon^4 + 5 \sup_{j} \left(x_j^\top \sscov x_j \right)^2. 
\end{align*}
Conveniently, we obtained an upper bound on $\EE[\xi_t^2 | x_{0:T}]$ that has the same form as the upper bound on $\EE [\zeta_t^2 | x_{0:T}]$ obtained in Section~\ref{sec:cov_ind_partI}. 

Now, we turn to upper bounding $\EE[\xi_t \xi_{t^\prime} | x_{0:T}]$ with $t^\prime > t$. There is a slight difference between the cases $t^\prime = t + 1$ and $t^\prime > t + 1$ separately. We have 
\begin{align}
\label{eq:cross_term2}
\EE[\xi_t \xi_{t + 1} | x_{0:T}] &= \EE \left[x_{t + 1}^\top A Z_t x_t x_{t + 2}^\top A Z_{t + 1} x_{t + 1} | x_{0:T}\right] + \sigma_\epsilon^2 \EE \left[x_t^\top \beta_t x_{t + 2}^\top A \beta_{t + 1}| x_{0:T}\right]\\
\EE[\xi_t \xi_{t^\prime} | x_{0:T}] &= \EE \left[x_{t + 1}^\top A Z_t x_t x_{t^\prime + 1}^\top A Z_{t^\prime} x_{t^\prime} | x_{0:T}\right] \label{eq:corss_term3} 
\end{align}
The term $\sigma_\epsilon^2 \EE \left[x_t^\top \beta_t x_{t + 2}^\top A \beta_{t + 1}| x_{0:T}\right]$ is equal to $\sigma_\epsilon^2 x_{t + 2}^\top A^2 \sscov x_t \leq \sigma_\epsilon^2 \|A^2 \sscov \| \sup_t \|x_t\|^2 \leq (\sigma_\epsilon^4 + \|A^2 \sscov \|^2 \sup_t \|x_t\|^4) / 2.$ 

We turn to bounding $\EE \left[x_{t + 1}^\top A Z_t x_t x_{t^\prime + 1}^\top A Z_{t^\prime} x_{t^\prime} | x_{0:T}\right]$ for $t^\prime > t$. Recalling that $Z_t = \beta_t \beta_t^\top - \sscov$ and the decomposition $\beta_{t^\prime} = \Astar^{t^\prime - t} \beta_t + u_{t, t^\prime}$ from Section~\ref{sec:cov_ind_partI}, we get
\begin{align}
\EE& \left[x_{t + 1}^\top A Z_t x_t x_{t^\prime + 1}^\top A Z_{t^\prime} x_{t^\prime} | x_{0:T}\right] 
= \EE \left[x_{t + 1}^\top A \beta_t\beta_t^\top x_t x_{t^\prime + 1}^\top \Astar^{t^\prime - t + 1} \beta_t \beta_t^\top (\Astar^\top)^{t^\prime - t} x_{t^\prime} | x_{0:T}\right] \\
&+ \EE \left[x_{t + 1}^\top A \beta_t \beta_t^\top x_t x_{t^\prime + 1}^\top A u_{t, t^\prime} u_{t, t^\prime}^\top x_{t^\prime} | x_{0:T}\right] - 
x_{t + 1}^\top A \sscov x_t x_{t^\prime + 1}^\top \Astar \sscov \ x_{t^\prime}
\end{align}
Since $\beta_t$ and $u_{t, t^\prime}$ are independent, the second term is equal to $x_{t + 1}^\top A \sscov x_t x_{t^\prime + 1}^\top A \sum_{j = 0}^{t^\prime - t - 1} A^j \Sigma_w (A^\top)^j x_{t^\prime}$. Now, since $\sscov = \sum_{j = 0}^\infty A^j \Sigma_w (A^\top)^j$, we get 
\begin{align}
\EE \left[x_{t + 1}^\top A Z_t x_t x_{t^\prime + 1}^\top A Z_{t^\prime} x_{t^\prime} | x_{0:T}\right] 
=& \EE \left[x_{t + 1}^\top A \beta_t\beta_t^\top x_t x_{t^\prime + 1}^\top \Astar^{t^\prime - t + 1} \beta_t \beta_t^\top (\Astar^\top)^{t^\prime - t} x_{t^\prime} | x_{0:T}\right] \\
&+ x_{t + 1}^\top A \sscov x_t x_{t^\prime + 1}^\top A \sum_{j = t^\prime - t}^{\infty} A^j \Sigma_w (A^\top)^j x_{t^\prime}.
\end{align}
The last term can be simplified by remembering that $\sum_{j = t^\prime - t}^{\infty} A^j \Sigma_w (A^\top)^j = A^{t^\prime - t} \sscov (A^\top)^{t^\prime - t}$. 

Now, to compute $D:= \EE \left[x_{t + 1}^\top A \beta_t\beta_t^\top x_t x_{t^\prime + 1}^\top \Astar^{t^\prime - t + 1} \beta_t \beta_t^\top (\Astar^\top)^{t^\prime - t} x_{t^\prime} | x_{0:T}\right]$ we first denote 
\begin{align*}
L_1 &= \frac{x_t x_{t + 1}^\top A  + (x_t x_{t + 1}^\top A )^\top}{2}\\
L_2 &= \frac{(\Astar^\top)^{t^\prime - t} x_{t^\prime} x_{t^\prime + 1}^\top \Astar^{t^\prime - t + 1} + ((\Astar^\top)^{t^\prime - t} x_{t^\prime} x_{t^\prime + 1}^\top \Astar^{t^\prime - t + 1})^\top}{2}.
\end{align*} 

With this notation we can write $D= \EE\left[\beta_t^\top L_1 \beta_t \beta_t^\top L_2 \beta_t | x_{0:T}\right]$. We introduced the matrices $L_1$ and $L_2$ because they are symmetric, which allows us to use the results of \citet{magnus1979expectation} to find 
\begin{align}
\EE\left[\beta_t^\top L_1 \beta_t \beta_t^\top L_2 \beta_t | x_{0:T}\right] = \tr(L_1 \sscov) \tr(L_2 \sscov) + 2\tr(L_1 \sscov L_2 \sscov). 
\end{align}
Now we apply the Cauchy-Schwartz inequality and the triangle inequality to find that 
\begin{align*}
|D| \leq 3 \|\sscov\|^2 \|L_1\|\|L_2\| &\leq 3 \|\sscov\|^2 \|A\| \tau(\Astar, \gamma)^2 \gamma^{2(t^\prime - t) + 1}   \sup_{t} \|x_t\|^4\\
&\leq 3 \|\sscov\|^2 \tau(\Astar, \gamma)^3 \gamma^{2(t^\prime - t + 1)}   \sup_{t} \|x_t\|^4.
\end{align*}

Hence, we can upper bound 
\begin{align*}
|\EE \left[x_{t + 1}^\top A Z_t x_t x_{t^\prime + 1}^\top A Z_{t^\prime} x_{t^\prime} | x_{0:T}\right]|  \leq 4 \|\sscov\|^2 \tau(\Astar, \gamma)^3 \gamma^{2(t^\prime - t + 1)}   \sup_{t} \|x_t\|^4.
\end{align*}
This step yielded the last upper bound we need on the covariance of the noise terms $\xi_t$. For future reference, let us list what we have proven:
\begin{align*}
\EE[\xi_t^2 | x_{0:T}] &\leq 3\sigma_\epsilon^4 + 5 \|\sscov\|^2 \sup_{j} \|x_j\|^4, \\
\EE[\xi_t \xi_{t + 1} | x_{0:T}] &\leq \frac{\sigma_\epsilon^4}{2} + \frac{1}{2}\|A^2\|^2 \|\sscov\|^2 \sup_j\|x_j\|^4  + 3 \|\sscov\|^2 \tau(\Astar, \gamma)^3 \gamma^{4}   \sup_{j} \|x_j\|^4 \\
&\leq \frac{\sigma_\epsilon^4}{2} + 4\|\sscov\|^2 \tau(\Astar, \gamma)^3 \gamma^{2}   \sup_{j} \|x_j\|^4,\\
\EE[\xi_t \xi_{t^\prime} | x_{0:T}] &\leq 4 \|\sscov\|^2 \tau(\Astar, \gamma)^3 \gamma^{2(t^\prime - t + 1)}   \sup_{t} \|x_t\|^4,
\end{align*}
where the last inequality holds for all $t^\prime > t + 1$. 

Let us put all these bounds together to upper bound $\sum_{i,j = 0}^{T - 1} \EE\left[\xi_i \xi_j|x_{0:T}\right] \phi_i \phi_j^\top$. As in Section~\ref{sec:cov_ind_partI}, we use the upper bound $uv^\top + vu^\top \preceq uu^\top + vv^\top$ and the upper bounds we derived on $\EE\left[\xi_i \xi_j|x_{0:T}\right]$. We find
\begin{align*}
\sum_{i,j = 0}^{T - 1} \EE\left[\xi_i \xi_j|x_{0:T}\right] \phi_i \phi_j^\top &\preceq \sum_{i = 0}^{T - 1} \left(\sum_{j = 0}^{T - 1}  |\EE\left[\xi_i \xi_j|x_{0:T}\right]| \right) \phi_i \phi_i^\top \\
&\preceq \left(4\sigma_\epsilon^4 + 8 \|\sscov\|^2 \frac{\tau(\Astar, \gamma)^3 }{1 - \gamma^2} \sup_{t} \|x_t\|^4\right)
\left(\sum_{i = 0}^{T - 1} \phi_i\phi_i^\top \right). 
\end{align*}
Therefore, we have shown that 
\begin{align*}
\Sigma_{e|x} \preceq \left(4\sigma_\epsilon^4 + 8 \|\sscov\|^2 \frac{\tau(\Astar, \gamma)^3 }{1 - \gamma^2} \sup_{t} \|x_t\|^4\right)
\left(\sum_{i = 0}^{T - 1} \phi_i\phi_i^\top \right)^{-1}.
\end{align*} 
As in Section~\ref{sec:cov_ind_partI}, we can upper bound $\sup_{j}\|x_j\|$ by $\sqrt{d} + \sqrt{2\log(2T/\delta)}$ with probability $1 - \delta$. This complete the proof. 

\section{Proof of Corollary~\ref{cor:cov}}

We recall that $\Ahat = \Mhat \sshat^{-1}$, where $\Mhat$ is the estimate of $\Astar \sscov$ produced by the OLS estimator \eqref{eq:cov_method}. Let $r$ denote the upper bound on the estimation error guaranteed in Theorem~\ref{thm:cov1}. Hence, $\|\sshat - \sscov\|_F \leq r$ and $\|\Mhat - \Astar \sscov \|_F \leq r$. 

Let us denote $\Delta_M =    \Astar \sscov - \Mhat$ and $\Delta_s = \sshat - \sscov $. 
Then, since our estimate of $A$ is $\Ahat =  \widehat{M}\sshat^{-1}$, we have
\begin{align*}
\Ahat - \Astar = \Mhat \sshat^{-1} -  \Astar \sscov \sscov^{-1}
\end{align*}
Therefore, 
\begin{align*}
 (\Ahat  - \Astar) \sshat &= \widehat{M} - (\widehat{M} + \Delta_M) \sscov^{-1} (\sscov + \Delta_s)= - \Delta_M + \Astar \Delta_s,
\end{align*}
which implies $\|\Ahat  - \Astar\|_F \leq r(1 + \|\Astar\|)\|\sshat^{-1}\|$ since $\|\Delta_M\|_F \leq r$ and $\|\Delta_s\|_F \leq r$. The conclusion follows once we require the number of samples $T$ to be large enough to guarantee that $r \leq \lambda_{\min}(\sscov)/2$, which ensures that $\lambda_{\min}(\sshat) \geq \lambda_{\min}(\sscov)/2$. 

\section{An extension of our analysis}
\label{app:features}

For simplicity, we have been assuming that $x_t \iid \Ncal(0, I_\statedim)$, $w_t \iid \Ncal(0, \Sigma_w)$, and $\epsilon_t \iid \Ncal(0, \sigma_\epsilon^2)$. However, our analysis can be easily extended to a more general setting. 

Firstly, the features $x_t$ do not need to be Gaussian nor do they have to be isotropic. The same analysis can be conducted as long as the features are sub-Gaussian. The covariance of the features can be assumed to be identity without loss of generality because when the features have some positive definite covariance $\Sigma_x$ we can whiten the features by $\Sigma_x^{-1/2} x_t$ and estimate the dynamics of the states $\widetilde{\beta}_t = \Sigma_x^{1/2} \beta_t$ instead. The transition matrix of $\widetilde{\beta}_t$ is given by $\Sigma_x^{1/2} \Astar \Sigma_x^{-1/2}$. 

The perceptive reader will note that our analysis only required that the noise terms $w_t$ and $\epsilon_t$ have zero means, are independent across time and of each other, and that they have finite fourth moments. Our analysis can be simply extended to this general setting.

\end{document}